\documentclass[12pt,reqno]{amsart}
\usepackage{amsmath,amssymb,verbatim,comment,color}
\usepackage[pagebackref,pdftex,hyperindex]{hyperref}

\usepackage[margin=3cm]{geometry}
\usepackage{color}
\newcommand{\C}{\mathbb{C}}

\newcommand{\R}{\mathbb{R}}

\newcommand{\cH}{\mathcal{H}}
\newcommand{\cI}{\mathcal{I}}
\newcommand{\cJ}{\mathcal{J}}
\newcommand{\cK}{\mathcal{K}}

\renewcommand{\a}{\alpha}
\renewcommand{\b}{\beta}
\newcommand{\g}{\gamma}
\renewcommand{\d}{\delta}
\newcommand{\e}{\varepsilon}

\renewcommand{\r}{\rho}

\newcommand{\p}{\partial}
\newcommand{\bp}{\bar{\partial}}
\newcommand{\dd}{\sqrt{-1}\partial \bar{\partial}}
\newcommand{\ddt}{\frac{d}{dt}}
\newcommand{\cf}{{\rm cf.\ }} 
\newcommand{\eg}{{\rm e.g.\ }} 
\newcommand{\ie}{{\rm i.e.\ }} 

\renewcommand{\Re}{\mathrm{Re}}
\renewcommand{\Im}{\mathrm{Im}}

\DeclareMathOperator{\Arg}{Arg}

\DeclareMathOperator{\Bl}{Bl}
\DeclareMathOperator{\dHYM}{dHYM}

\DeclareMathOperator{\HC}{HC}

\renewcommand{\leq}{\leqslant}
\renewcommand{\geq}{\geqslant}

\renewcommand{\hat}{\widehat}
\renewcommand{\tilde}{\widetilde}

\numberwithin{equation}{section}       % Number formulas within sections

\newtheorem{prop} {Proposition} [section]
\newtheorem{thm}[prop] {Theorem} 

\newtheorem{lem}[prop] {Lemma}
\newtheorem{cor}[prop]{Corollary}

\theoremstyle{remark}

\title{Collapsing of the line bundle mean curvature flow on K\"ahler surfaces} 
\date{\today}

\author[R. Takahashi]{Ryosuke Takahashi}
\address{Faculty of Mathematics\\
 Kyushu University\\
744\\
Motooka\\
Nishi-ku\\
Fukuoka\\
819-0395\\
 JAPAN}
\email{rtakahashi@math.kyushu-u.ac.jp}


\begin{document}
\maketitle
\begin{abstract}
We study the line bundle mean curvature flow on K\"ahler surfaces under the hypercritical phase and a certain semipositivity condition. We naturally encounter such a condition when considering the blowup of K\"ahler surfaces. We show that the flow converges smoothly to a singular solution to the deformed Hermitian--Yang--Mills equation away from a finite number of curves of negative self-intersection on the surface. As an application, we obtain a lower bound of a Kempf--Ness type functional on the space of potential functions satisfying the hypercritical phase condition.
\end{abstract}
%==============Section 1==========================
\section{Introduction}
Let $X$ a compact $n$-dimensional complex manifold with a fixed K\"ahler form $\a$ and a closed real $(1,1)$-form $\hat{F}$. For any smooth function $\phi \in C^\infty(X;\R)$ we set $F_\phi:=\hat{F}+\dd \phi$ (but we often drop the subscript $\phi$ when the dependence is clear). We assume that the integral
\[
Z:=\int_X (\a+\sqrt{-1} F_\phi)^n
\]
does not vanish. Then we can define the unique number $e^{\sqrt{-1} \hat{\Theta}} \in U(1)$ by requiring $e^{-\sqrt{-1} \hat{\Theta}} Z \in \R_{>0}$. We say that the function $\phi$ satisfies the {\it deformed Hermitian--Yang--Mills (dHYM) equation} if
\begin{equation} \label{dHYM}
\Im \bigg( e^{-\sqrt{-1} \hat{\Theta}} (\a+\sqrt{-1} F_\phi)^n \bigg)=0.
\end{equation}
This equation first appeared in the physics literature \cite{MMMS00}, and \cite{LYZ01} from mathematical side as the mirror object to a special Lagrangian in the setting of semi-flat mirror symmetry, and is studied actively in recent years (\eg \cite{Che19,CJY20,CY18,HY19,JY17,Pin19,SS19}). Let $\lambda_i$ be the eigenvalues of the endomorphism $F_{i \bar{j}} \a^{k \bar{j}}$ and define the {\it Lagrangian phase} by
\[
\Theta_\a:=\sum_{i=1}^n \arctan \lambda_i.
\]
In practice, we often write the dHYM equation \eqref{dHYM} in terms of $\Theta_\a$ as
\[
\Theta_\a=\hat{\Theta} \quad \text{(mod. $2\pi$)}.
\]
In order to get the solution, Jacob--Yau \cite{JY17} introduced the following parabolic evolution equation, called the {\it line bundle mean curvature flow}\footnote{In \cite{JY17}, they made an assumption that the form $F$ arises as curvature of a fiber metric on a holomorphic line bundle. But this assumption is only for aesthetic purposes and not essentially used in their paper.} (LBMCF):
\begin{equation} \label{LBMCF}
\ddt \phi=\Theta_\a-\hat{\Theta}.
\end{equation}
The LBMCF is the gradient flow of the {\it volume functional}
\[
V(\phi):=\int_X \bigg| \frac{(\a+\sqrt{-1}F_\phi)^n}{\a^n} \bigg| \a^n, \quad \phi \in C^\infty(X;\R).
\]
In what follows, we only consider the case $\hat{\Theta}>0$ (but the similar arguments also work for the case $\hat{\Theta}<0$). We say that $\phi \in C^\infty(X;\R)$ is {\it hypercritical} if it lies in the set
\[
\cH_{\HC}:=\bigg\{\phi \in C^\infty(X;\R) \bigg| \Theta_\a(\phi)>(n-1)\frac{\pi}{2} \bigg\}.
\]
Since $\arctan(\cdot)$ takes values in $(-\frac{\pi}{2}, \frac{\pi}{2})$, the condition $\phi \in \cH_{\HC}$ yields that $\lambda_i>c$ for all $i$, where the constant $c>0$ depends only on $\inf_X \Theta_\a(\phi)$. In particular, the form $F_\phi$ must be K\"ahler. In \cite{JY17}, they showed that if the K\"ahler metric $\a$ has non-negative orthogonal bisectional curvature and the initial data $\phi_0$ satisfies $\phi_0 \in \cH_{\HC}$, then the flow converges to the solution to the dHYM equation. While the limiting behavior of the flow is of independent interest even when $X$ does not admit the solution. The reason is that in the absence of the solution, the flow fails to converge, which provides an optimal way of detecting instability (for instance, see \cite{CHT17,CSW18,DS16}). Also, as studied in \cite{CY18}, the dHYM equation has a moment map interpretation which interpolates that of Hermitian--Yang--Mills equation and $J$-equation \cite{Che00,Don99}. For this reason, one can expects that certain aspects of the $J$-flow carry over to the LBMCF.

In this paper, we study the limiting behavior of the LBMCF on K\"ahler surfaces under a certain semipositivity assumption, motivated by a work on the $J$-flow \cite{FLSW14}. When $n=2$, it was shown in \cite[Theorem 1.2]{JY17} that the existence of the solution to \eqref{dHYM} is equivalent to the K\"ahler condition of the associated cohomology class
\begin{equation} \label{JYcriterion}
\cot \hat{\Theta} [\a]+[\hat{F}]>0.
\end{equation}
We relax this condition as
\begin{equation}
\cot \hat{\Theta} [\a]+[\hat{F}] \geq 0,
\end{equation}
\ie the class $\cot \hat{\Theta} [\a]+[\hat{F}]$ is represented by some smooth closed semipositive real $(1,1)$-form. We remark that if $\hat{\Theta}>\frac{\pi}{2}$, then the semipositivity $\cot \hat{\Theta} [\a]+[\hat{F}] \geq 0$ implies that $[\hat{F}]$ is a K\"ahler class. The following is our main theorem:
\begin{thm} \label{convf}
Let $X$ be a compact complex surface with a K\"ahler form $\a$ and a closed real $(1,1)$-form $\hat{F}$ such that
\[
\cot \hat{\Theta} \a+\hat{F} \geq 0, \quad \hat{\Theta}>\frac{\pi}{2}.
\]
Assume $\phi_0 \in \cH_{\HC}$ (this condition assures that $F_{\phi_t}$ is K\"ahler for all $t$). Then there exist a finite number of curves $C_i$ ($i=1,\ldots,N$) on $X$ of negative self-intersection such that the line bundle mean curvature flow $\phi_t$ converges to a bounded function $\phi_\infty$ in $C_{\rm loc}^\infty(X \backslash \bigcup_i C_i)$ as $t \to \infty$, where $F_{\phi_\infty}:=\hat{F}+\dd \phi_\infty$ is a K\"ahler current, which is smooth on $X \backslash \bigcup_i C_i$ and satisfies the deformed Hermitian--Yang--Mills equation
\begin{equation} \label{dHYMtwo}
\Im \bigg( e^{-\sqrt{-1} \hat{\Theta}} (\a+\sqrt{-1} F_{\phi_\infty})^2 \bigg)=0
\end{equation}
on $X \backslash \bigcup_i C_i$. Moreover, the convergence $F_{\phi_t} \to F_{\phi_\infty}$ as well as \eqref{dHYMtwo} holds on $X$ in the sense of currents.
\end{thm}
A similar semipositivity condition and collapsing property of the $J$-flow were studied in \cite{FLSW14}. The proof of Theorem \ref{convf} proceeds in the same way as \cite{FLSW14,SW08}, \ie we rewrite \eqref{dHYM} as a degenerate complex Monge--Amp\`ere equation \eqref{dgMA}, and apply a result of \cite{EGZ09} to get a singular solution $\psi$. The function $\psi$ is used to obtain the $C^0$-estimate of $\phi_t$. Unlike the $J$-flow, we need a careful choice of the initial metric $F_{\phi_0}$ since the operator we study fails to be concave in general. Following \cite{JY17,CJY20}, we require the assumption $\phi_0 \in \cH_{\HC}$ that assures the concavity of the operator along the LBMCF, and hence the standard Evans--Krylov theory \cite{Kry82,Wan12} does apply. Also we can construct some examples satisfying all of the assumptions in Theorem \ref{convf}, as a small deformation of the pullback of the ``trivial solution'' $(\a,m\a)$ ($m>1$) on the blowup of K\"ahler surfaces.

As an application, we obtain a lower bound of a Kempf--Ness type functional. According to \cite{CY18}, we set
\[
\cH:=\bigg\{ \phi \in C^\infty(X;\R) \bigg| \hat{\Theta}-\frac{\pi}{2} < \Theta_\a(\phi) < \hat{\Theta}+\frac{\pi}{2} \bigg\},
\]
and define the {\it $\cJ$-functional} by the variational formula
\[
\d \cJ(\d \phi)=-\int_X \d \phi \Im \big( e^{-\sqrt{-1}\hat{\Theta}}(\a+\sqrt{-1} F_\phi)^n \big).
\]
So $\phi \in \cH$ is a critical point of $\cJ$ if and only if $\phi$ solves the dHYM equation $\Theta_\a(\phi)=\hat{\Theta}$. The $\cJ$-functional plays a role of the Kempf--Ness functional in an infinite dimensional GIT picture (see \cite[Section 2]{CY18} for more details). When $\hat{\Theta} \in ((n-1)\frac{\pi}{2},n \frac{\pi}{2})$, one can easily see that $\cH_{\HC} \subset \cH$. In the absence of the solution to \eqref{dHYM}, it is important to study the boundary behavior of the $\cJ$-functional. As a corollary of Theorem \ref{convf}, we obtain the following:
\begin{cor} \label{Jbound}
Let $X$ be a compact complex surface with a K\"ahler form $\a$ and a closed real $(1,1)$-form $\hat{F}$ such that
\[
\cot \hat{\Theta} \a+\hat{F} \geq 0, \quad \hat{\Theta}>\frac{\pi}{2}.
\]
Then the $\cJ$-functional is bounded from below on $\cH_{\HC}$.
\end{cor}
This is also an analogy for the $J$-equation case \cite[Corollary 3.3]{FLSW14}. We expect that this result can be improved for the whole space $\cH$.

Now we briefly explain the organization of this paper. In Section \ref{backnot}, we first recall some background and fix notations that we will use in later arguments. One can find all of these items in \cite{CXY17,CY18,JY17}. Then based on \cite{JY17}, we establish the long time existence result of \eqref{LBMCF} under the hypercritical phase condition for arbitrary dimension in Section \ref{ltimef}. In the later part of the paper, we focus on the case $n=2$. In Section \ref{blowup}, we explore some examples on the blowup of K\"ahler surfaces. Finally, in Section \ref{convflow}, we begin with a brief review of degenerate complex Monge--Amp\`ere equations. Then we prove Theorem \ref{convf} and Corollary \ref{Jbound}.
%==============Section 2==========================
\section{Background and notation} \label{backnot}
Let $X$ be a compact $n$-dimensional complex manifold with a K\"ahler form $\a$ and a closed real $(1,1)$-form $\hat{F}$.
\subsection{Formulas}
We define a $\C^\ast$-valued function $\zeta$ by
\[
\zeta:=\frac{(\a+\sqrt{-1}F_\phi)^n}{\a^n}.
\] 
Then the function $\zeta$ is related to $\Theta_\a$ as follows
\[
\zeta=\prod_i (1+\sqrt{-1}\lambda_i)=v e^{\sqrt{-1} \Theta_\a},
\]
\[
v:=|\zeta|=\sqrt{\prod_i(1+\lambda_i^2)},
\]
where $\lambda_i$ denote the eigenvalues of $F_\phi$ with respect to $\a$. Next we recall the variation formula of $\Theta_\a$. We define a Hermitian metric $\eta$ on $T^{1,0}X$ by
\[
\eta_{i \bar{j}}=\a_{i \bar{j}}+F_{i \bar{\ell}} \a^{k \bar{\ell}} F_{k \bar{j}}.
\]
Also, for any $f \in C^\infty(X;\R)$, we define
\[
\Delta_\eta f:=\eta^{i \bar{j}} \p_i \p_{\bar{j}} f.
\]
Then the variation of the Lagrangian phase $\Theta_\a$ is given by
\begin{equation} \label{vfLph}
\d \Theta_\a(\d \phi)=\Delta_\eta \d \phi.
\end{equation}
\subsection{Functionals}
We recall the definitions of functionals introduced in \cite{CY18} and study some basic properties of them. We define the {\it Calabi--Yau functional} $CY_\C$ by
\[
CY_\C(\phi):=\frac{1}{n+1} \sum_{j=0}^n \int_X \phi(\a+\sqrt{-1}F_\phi)^j \wedge (\a+\sqrt{-1}\hat{F})^{n-j}, \quad \phi \in \cH.
\]
This is a $\C$-valued functional. The variational formula of $CY_\C$ is given by
\[
\d CY_\C(\d \phi)=\int_X \d \phi (\a+\sqrt{-1} F_\phi)^n.
\]
Also, for $\phi \in \cH$ we set
\[
\cI(\phi):=\Re \big(e^{-\sqrt{-1}(n-1)\frac{\pi}{2}} CY_\C(\phi) \big),
\]
\[
\cJ(\phi):=-\Im \big(e^{-\sqrt{-1}\hat{\Theta}} CY_\C(\phi) \big).
\]
Then the variational formula for $CY_\C$ yields that
\[
\d \cI(\d \phi)=\int_X \d \phi \Re \big( e^{-\sqrt{-1}(n-1)\frac{\pi}{2}} (\a+\sqrt{-1} F_\phi)^n \big),
\]
\[
\d \cJ(\d \phi)=-\int_X \d \phi \Im \big( e^{-\sqrt{-1} \hat{\Theta}} (\a+\sqrt{-1} F_\phi)^n \big).
\]
We use the following simple observations:
\begin{prop} \label{monotonicity}
Assume $\hat{\Theta} \in ((n-1)\frac{\pi}{2},n \frac{\pi}{2})$. For all $\phi \in \cH$ and $c \in \R$, we have
\begin{enumerate}
\item $CY_\C(\phi+c)=CY_\C(\phi)+cZ$.
\item $\cI(\phi+c)=\cI(\phi)+c|Z| \cos \big(\hat{\Theta}-(n-1)\frac{\pi}{2} \big)$.
\item $\cJ(\phi+c)=\cJ(\phi)$.
\end{enumerate}
Moreover, for the line bundle mean curvature flow $\phi_t$ with $\phi_0 \in \cH_{\HC}$, we have
\begin{enumerate}
\setcounter{enumi}{3}
\item $\cI(\phi_t)$ is monotonically decreasing.
\item $\cJ(\phi_t)$ is monotonically decreasing.
\end{enumerate}
\end{prop}
\begin{proof}
The property (1) is trivial. The properties (2) and (3) follow from (1). Now let $\phi_t$ be the LBMCF with $\phi_0 \in \cH_{\HC}$. Then the flow $\phi_t$ exists and lies in $\cH_{\HC}$ for all positive time (see Section \ref{ltimef}). From the variational formula of $\cI$, we compute
\begin{eqnarray*}
\ddt \cI(\phi_t)&=&\int_X (\Theta_\a-\hat{\Theta}) \Re \big( e^{-\sqrt{-1}(n-1)\frac{\pi}{2}} (\a+\sqrt{-1} F_\phi)^n \big)\\
&=& \int_X \bigg( \Theta_\a-(n-1)\frac{\pi}{2} \bigg) \Re \big( e^{-\sqrt{-1}(n-1)\frac{\pi}{2}} (\a+\sqrt{-1} F_\phi)^n \big)\\
&-& \int_X \bigg( \hat{\Theta}-(n-1)\frac{\pi}{2} \bigg) \Re \big( e^{-\sqrt{-1}(n-1)\frac{\pi}{2}} (\a+\sqrt{-1} F_\phi)^n \big) \\
&=& \int_X \bigg( \Theta_\a-(n-1)\frac{\pi}{2} \bigg) \Re \big( e^{-\sqrt{-1}(n-1)\frac{\pi}{2}} (\a+\sqrt{-1} F_\phi)^n \big) \\
&-& \bigg( \hat{\Theta}-(n-1)\frac{\pi}{2} \bigg) |Z| \cos \bigg(\hat{\Theta}-(n-1)\frac{\pi}{2} \bigg).
\end{eqnarray*}
Since $(n-1)\frac{\pi}{2}<\Theta_\a<n\frac{\pi}{2}$, we observe that $\Re \big( e^{-\sqrt{-1}(n-1)\frac{\pi}{2}} (\a+\sqrt{-1} F_\phi)^n \big)$ is a positive measure with volume $|Z| \cos (\hat{\Theta}-(n-1)\frac{\pi}{2})$. So by Jensen's inequality, we compute the first term as
\begin{eqnarray*}
& & \int_X \bigg( \Theta_\a-(n-1)\frac{\pi}{2} \bigg) \Re \big( e^{-\sqrt{-1}(n-1)\frac{\pi}{2}} (\a+\sqrt{-1} F_\phi)^n \big)\\
& & \leq |Z| \cos \bigg(\hat{\Theta}-(n-1)\frac{\pi}{2} \bigg) \cdot \arctan \left[ \int_X \tan \bigg( \Theta_\a-(n-1)\frac{\pi}{2} \bigg) \frac{\Re \big( e^{-\sqrt{-1}(n-1)\frac{\pi}{2}} (\a+\sqrt{-1} F_\phi)^n \big)}{|Z| \cos \big(\hat{\Theta}-(n-1)\frac{\pi}{2} \big)} \right] \\
& & =|Z| \cos \bigg(\hat{\Theta}-(n-1)\frac{\pi}{2} \bigg) \cdot \arctan \left[ \int_X \tan \bigg( \Theta_\a-(n-1)\frac{\pi}{2} \bigg) \frac{v \cos \big( \Theta_\a-(n-1)\frac{\pi}{2} \big)}{|Z| \cos \big(\hat{\Theta}-(n-1)\frac{\pi}{2} \big)} \a^n \right] \\
& & =|Z| \cos \bigg(\hat{\Theta}-(n-1)\frac{\pi}{2} \bigg) \cdot \arctan \left[ \frac{\int_X v \sin \big( \Theta_\a-(n-1)\frac{\pi}{2} \big) \a^n}{|Z| \cos \big(\hat{\Theta}-(n-1)\frac{\pi}{2} \big)} \right].
\end{eqnarray*}
Since
\[
\int_X v \sin \bigg( \Theta_\a-(n-1)\frac{\pi}{2} \bigg) \a^n=\int_X \Im \big( e^{-\sqrt{-1}(n-1) \frac{\pi}{2}}(\a+\sqrt{-1}F)^n \big)=|Z| \sin \bigg( \hat{\Theta}-(n-1)\frac{\pi}{2} \bigg),
\]
we obtain a bound
\[
\int_X \bigg( \Theta_\a-(n-1)\frac{\pi}{2} \bigg) \Re \big( e^{-\sqrt{-1}(n-1)\frac{\pi}{2}} (\a+\sqrt{-1} F_\phi)^n \big) \leq \bigg( \hat{\Theta}-(n-1)\frac{\pi}{2} \bigg) |Z| \cos \bigg(\hat{\Theta}-(n-1)\frac{\pi}{2} \bigg).
\]
This shows $\ddt \cI(\phi_t) \leq 0$, so the property (4) holds. Taking $-\frac{\pi}{2}<\Theta_\a-\hat{\Theta}<\frac{\pi}{2}$ into account, we can compute $\ddt \cJ(\phi_t)$ similarly as
\begin{eqnarray*}
\ddt \cJ(\phi_t) &=& -\int_X (\Theta_\a-\hat{\Theta}) \Im \big( e^{-\sqrt{-1}\hat{\Theta}} (\a+\sqrt{-1} F_\phi)^n \big)\\
&=& -\int_X (\Theta_\a-\hat{\Theta}) \cdot v \sin (\Theta_\a-\hat{\Theta}) \a^n \\
&\leq& 0.
\end{eqnarray*}
Thus we have (5). This completes the proof.
\end{proof}
\subsection{The case $n=2$}
We recall some special properties when $n=2$. Since
\begin{equation} \label{compZ}
Z=[\a]^2-[\hat{F}]^2+2\sqrt{-1} [\a] \cdot [\hat{F}],
\end{equation}
as pointed out in \cite[page 13]{CXY17}, one can check that $Z \in \C \backslash (-\infty,0]$ by using the Hodge index theorem. So the condition $Z \neq 0$ is satisfied automatically (this is not true for $n \geq 3$ as argued in \cite[Lemma 2.1]{CXY17}). Similarly, since
\[
\zeta=1-\lambda_1 \lambda_2+\sqrt{-1}(\lambda_1+\lambda_2),
\]
if $\zeta=\lambda_1+\lambda_2$ vanishes, then the real part $\Re \zeta=1-\lambda_1 \lambda_2=1+\lambda_1^2$ must be positive. This shows that the function $\zeta$ takes values in $\C \backslash (-\infty,0]$, so the branch cuts of the argument of $Z$ and $\zeta$ are specified to $(-\pi, \pi)$. In particular, we have $\Theta_\a=\Arg \zeta$, where $\Arg \colon \C \backslash (-\infty,0] \to (-\pi, \pi)$ denotes the principal argument. Also by \eqref{compZ}, the argument $\hat{\Theta} \in (-\pi, \pi)$ is determined by the formula
\begin{equation} \label{decot}
\cot \hat{\Theta}=\frac{[\a]^2-[\hat{F}]^2}{2 [\a] \cdot [\hat{F}]}.
\end{equation}
%==============Section 3==========================
\section{Long time existence} \label{ltimef}
In this section, we will show the following:
\begin{thm} \label{ltime}
Let $X$ be a compact complex manifold with a K\"ahler form $\a$ and a closed real $(1,1)$-form $\hat{F}$. Assume $\phi_0 \in \cH_{\HC}$. Then the line bundle mean curvature flow exists for all positive time.
\end{thm}
This theorem follows directly from the argument in \cite[Section 5]{JY17} (although it is not mentioned explicitly in their paper). We will give a proof here for the sake of completeness. First we remark that the evolution equation \eqref{LBMCF} is parabolic by \eqref{vfLph}, so the flow exists for a short time $t \in [0,T)$, where $T$ denotes the maximum existence time of the flow. Also we know that
\[
\ddt \Theta_\a=\Delta_\eta \Theta_\a.
\]
Applying the standard maximum principle to this, we obtain the following:
\begin{lem} \label{hpufl}
We have
\[
\inf_X \Theta_\a(\phi_0) \leq \Theta_\a(\phi_t) \leq \sup_X \Theta_\a(\phi_0)
\]
for all $t \in [0,T)$. In particular, the hypercritical phase condition is preserved under the line bundle mean curvature flow.
\end{lem}
Now we assume $\phi_0 \in \cH_{\HC}$. The above lemma already implies a uniform lower bound $\lambda_i>c$ for some constant $c>0$ (depending only on $\a$ and $\phi_0$). Also combining with \eqref{LBMCF}, we obtain the following:
\begin{lem}
We have $\|\phi_t\|_{C^0} \leq C(T+1)$, where the constant $C>0$ depends only on $\a$ and $\phi_0$.
\end{lem}
In order to obtain an upper bound of $\lambda_i$, we need to compute the evolution equation of $v$ as in the following lemma:
\begin{lem} \label{sderiv}
We have $v \leq e^{C(T+1)}$, where the constant $C>0$ depends only on $\phi_0$ and a lower bound of the orthogonal bisectional curvature of $\a$.
\end{lem}
\begin{proof}
First we recall the evolution equation of $v$ (\cf \cite[page 887]{JY17}):
\begin{eqnarray*}
\bigg( \ddt-\Delta_\eta \bigg) \log v&=&\eta^{j \bar{k}} F_{p \bar{k}} g^{p \bar{q}} \eta^{\ell \bar{m}} [\nabla_{\bar{q}}, \nabla_\ell]F_{j \bar{m}}-|\nabla F|_\eta^2 \\
&-& \eta^{j \bar{k}} F_{p \bar{k}} g^{p \bar{q}} \eta^{\ell \bar{s}} F_{r \bar{s}} g^{r \bar{b}} \nabla_{\bar{q}} F_{a \bar{b}} \eta^{a \bar{m}} \nabla_\ell F_{j \bar{m}},
\end{eqnarray*}
where $\nabla$ denotes the covariant derivative with respect to $\a$. In normal coordinates, the last term is given by
\[
-\sum_{\ell, j, r} \frac{\lambda_j \lambda_\ell |\nabla_j F_{\ell \bar{r}}|^2}{(1+\lambda_j^2)(1+\lambda_r^2)(1+\lambda_\ell^2)},
\]
which is non-positive since the eigenvalues $\lambda_i$ are all positive. Also, using a lower bound of the orthogonal bisectional curvature of $\a$ and the symmetry $R_{\ell \bar{j} j \bar{\ell}}=R_{j \bar{\ell} \ell \bar{j}}$, the first term is computed as
\begin{eqnarray*}
\eta^{j \bar{k}} F_{p \bar{k}} g^{p \bar{q}} \eta^{\ell \bar{m}} [\nabla_{\bar{q}}, \nabla_\ell]F_{j \bar{m}} &=& \eta^{j \bar{j}} F_{j \bar{j}} \eta^{\ell \bar{\ell}} R_{\ell \bar{j} j}{}^\ell F_{\ell \bar{\ell}}-\eta^{j \bar{j}} F_{j \bar{j}} \eta^{\ell \bar{\ell}} R_{\ell \bar{j}}{}^{\bar{j}}{}_{\bar{\ell}} F_{j \bar{j}}\\
&=& \sum_{j,\ell} \frac{(\lambda_j \lambda_\ell-\lambda_j^2) R_{\ell \bar{j} j \bar{\ell}}}{(1+\lambda_j^2)(1+\lambda_\ell^2)}\\
&=& \sum_{j<\ell} \frac{(2 \lambda_j \lambda_\ell-\lambda_j^2-\lambda_\ell^2)R_{\ell \bar{j} j \bar{\ell}}}{(1+\lambda_j^2)(1+\lambda_\ell^2)}\\
&=& -\sum_{j<\ell} \frac{(\lambda_j-\lambda_\ell)^2 R_{\ell \bar{j} j \bar{\ell}}}{(1+\lambda_j^2)(1+\lambda_\ell^2)}\\
&\leq & C \sum_{j<\ell} \frac{(\lambda_j-\lambda_\ell)^2}{(1+\lambda_j^2)(1+\lambda_\ell^2)}\\
&\leq & 2C \sum_{j<\ell} \frac{\lambda_j^2+\lambda_\ell^2}{(1+\lambda_j^2)(1+\lambda_\ell^2)}\\
&\leq& Cn(n-1).
\end{eqnarray*}
Hence we have
\[
\bigg( \ddt-\Delta_\eta \bigg) \log v \leq Cn(n-1).
\]
Applying the maximum principle, we obtain the desired result.
\end{proof}
The above lemma implies that $F$ has a two-sided bound
\begin{equation} \label{twosideF}
C_T^{-1} \a<F_{\phi_t}<C_T \a
\end{equation}
for some constant $C_T>0$ which depends only on $\a$, $\phi_0$ and $T$. Now we recall the argument in \cite[page 889]{JY17}. From \eqref{twosideF}, we find that the operator $\Theta_\a$ is uniformly elliptic and concave. So we can apply the Evans--Krylov estimate \cite{Kry82,Wan12} to obtain $C^{2,\g}$-bounds of $\phi_t$ for some $\g \in (0,1)$. Then the standard bootstrapping argument shows that all the higher order derivatives of $\phi_t$ is uniformly controlled on $[0,T)$, and hence we conclude $T=\infty$. This completes the proof of Theorem \ref{ltime}.
%==============Section 4==========================
\section{Examples: the blowup of K\"ahler surfaces} \label{blowup}
Let $X$ be a compact complex surface with a K\"ahler form $\a$. For $m>1$, we consider a pair $([\a], m[\a])$. One can easily check that this pair admits a trivial solution $(\a, m\a)$ to \eqref{dHYM} whose associated constant argument $\hat{\Theta}>\frac{\pi}{2}$ is determined by
\[
\cot \hat{\Theta}=\frac{[\a]^2-m^2[\a]^2}{2 [\a] \cdot m[\a]}=\frac{1-m^2}{2m}.
\]
Let $z \in X$ be a point, and denote by $\pi \colon \Bl_z X \to X$ the blowup of $X$ at $z$ with exceptional divisor $E$. Now we will construct a pair of K\"ahler classes on $\Bl_z X$ satisfying the assumption in Theorem \ref{convf} as a small deformation of $(\pi^\ast [\a], m \pi^\ast [\a])$. We note that for small numbers $s,t>0$, $(\pi^\ast [\a]-s[E], m \pi^\ast [\a]-t[E])$ defines a pair of K\"ahler classes. Let $\hat{\Theta}(s,t)$ be the associated constant argument of this pair. Since
\[
[E]^2=-1, \quad \pi^\ast [\a] \cdot [E]=0,
\]
the argument $\hat{\Theta}(s,t)$ is given by
\[
\cot \hat{\Theta}(s,t)=\frac{L-s^2-(m^2 L-t^2)}{2(mL-st)}=\frac{(1-m^2)L-s^2+t^2}{2(mL-st)},
\]
where we set $L:=[\a]^2>0$. We try to find $(s,t)$ such that
\[
\cot \hat{\Theta}(s,t) (\pi^\ast [\a]-s[E])+m \pi^\ast [\a]-t[E] \in \R_{>0} \pi^\ast [\a].
\]
This condition is clearly satisfied when $(s,t)=(0,0)$ since
\[
\cot \hat{\Theta}+m=\frac{1+m^2}{2m}>0.
\]
For general $(s,t)$, this holds if and only if
\[
\begin{cases}
\cot \hat{\Theta}(s,t)+m>0 \\
\cot \hat{\Theta}(s,t)s+t=0.
\end{cases}
\]
The first condition is open with respect to $(s,t)$. The second condition is equivalent to $G(s,t)=0$, where
\[
G(s,t):=s[(1-m^2)L-s^2+t^2]+2t(mL-st).
\]
Since $\frac{\p G}{\p s} \big|_{(s,t)=(0,0)}=(1-m^2)L<0$ and $\frac{\p G}{\p t} \big|_{(s,t)=(0,0)}=2mL>0$, applying the implicit function theorem, we know that $t=t(s)$ is a smooth function of $s$ near $(0,0)$ with $G(s,t(s))=0$ and $\frac{dt}{ds}|_{s=0}=\frac{m^2-1}{2m}>0$. So we obtain the pair of K\"ahler class $(\pi^\ast [\a]-s[E], m \pi^\ast [\a]-t(s)[E])$ on $\Bl_z X$ for sufficiently small $s>0$ whose associated form $\cot \hat{\Theta} (s,t(s)) (\pi^\ast [\a]-s[E])+m \pi^\ast [\a]-t(s)[E]$ lies in the ray $\R_{>0} \pi^\ast [\a]$ and is clearly semipositive. Also we note that $\hat{\Theta} (s,t(s))>\frac{\pi}{2}$ as long as $s>0$ is small. We fix $s>0$ and set $\hat{\Psi}:=\hat{\Theta}(s,t(s))$.

In general, it seems to be hard to check whether there exists a K\"ahler metric satisfying the hypercritical phase condition in a given cohomology class even if we know that the associated constant argument is greater than $\frac{\pi}{2}$. However, we can check that the above example $(\Upsilon, \Gamma):=(\pi^\ast [\a]-s[E], m \pi^\ast [\a]-t(s)[E])$ does include such a metric as follows: first we note that $\Upsilon$ is not proportional to $\Gamma$ (since, if so, the pair $(\Upsilon, \Gamma)$ must admit a trivial solution to \eqref{dHYM}). We set
\[
M:=\Upsilon^2, \quad N:=\Gamma^2, \quad S:=\Upsilon \cdot \Gamma
\]
for simplifying notations. For a small $\d>0$, we consider the pair of K\"ahler classes $(\Upsilon,(1-\d)\Gamma)$ with the associated constant argument $\hat{\Psi}(\d)$ given by
\[
\cot \hat{\Psi}(\d)=\frac{M-(1-\d)^2N}{2(1-\d)S}.
\]
Let $\cK$ be a subcone of the K\"ahler cone of $\Bl_z X$ cut out by a subspace $W \subset H^{1,1}(\Bl_z X;\R)$ spanned by $\Upsilon$ and $\Gamma$. Then the class $\cot \hat{\Psi}(\d) \Upsilon+(1-\d)\Gamma$ ($\d>0$) defines a rational curve in $W$. We compute the velocity vector of this curve at $\d=0$ as
\[
\frac{d}{d \d}\big( \cot \hat{\Psi}(\d) \Upsilon+(1-\d)\Gamma \big)|_{\d=0}=\frac{d}{d \d} \cot \hat{\Psi}(\d)|_{\d=0} \Upsilon-\Gamma.
\]
Since
\[
\frac{d}{d \d} \cot \hat{\Psi}(\d)|_{\d=0}=\frac{M+N}{2S}>0, \quad \cot \hat{\Psi}=\frac{M-N}{2S}<0, \quad \bigg| \frac{M-N}{2S} \bigg|<\bigg| \frac{M+N}{2S} \bigg|,
\]
one can easily see that this is an inward vector of $\cK$ at $\cot \hat{\Psi} \Upsilon+\Gamma \in \p \cK$. Thus the class $\cot \hat{\Psi}(\d) \Upsilon+(1-\d)\Gamma$ lies in $\cK$ for sufficiently small $\d>0$, and hence is K\"ahler. Applying the criterion \eqref{JYcriterion}, we know that for any fixed K\"ahler form $\b \in \Upsilon$, the pair $(\Upsilon,(1-\d)\Gamma)$ admits a solution to \eqref{dHYM} with respect to $\b$ for sufficiently small $\d>0$, say $F_{\dHYM,\d} \in (1-\d)\Gamma$. Take any smooth representative $\tilde{F} \in \Gamma$, and consider $F_\d:=F_{\dHYM,\d}+\d \tilde{F} \in \Gamma$. Let $\lambda_{\dHYM,\d,i}$ (resp. $\lambda_{\d,i}$) denotes the eigenvalue of $F_{\dHYM,\d}$ (resp. $F_\d$) in decreasing order with respect to $\b$. Now we take a local coordinates, and regard the endomorphisms $(F_\d)_{i \bar{j}} \a^{k \bar{j}}$, $(F_{\dHYM,\d})_{i \bar{j}} \a^{k \bar{j}}$, $\tilde{F}_{i \bar{j}} \a^{k \bar{j}}$ as matrix-valued functions on it.

Since the map $A \mapsto (\lambda_1(A), \lambda_2(A))$ from Hermitian $2 \times 2$ matrices to the eigenvalues in decreasing order is Lipschitz and $\tilde{F}$ is a fixed form, we observe that
\[
\|\lambda_{\d,i}-\lambda_{\dHYM,\d,i}\|_{C^0}=O(\d), \quad i=1,2
\]
as $\d \to 0$. Combining with the Lipschitz continuity of $\arctan (\cdot)$, we have an estimate for the corresponding Lagrangian phase functions
\[
\|\Theta_\b(F_\d)-\hat{\Psi}(\d) \|_{C^0}=O(\d)
\]
as $\d \to 0$. Since the constant argument $\hat{\Psi}(\d)$ converges to $\hat{\Psi}>\frac{\pi}{2}$ as $\d \to 0$, we conclude that $\Theta_\b(F_\d)$ is hypercritical for sufficiently small $\d>0$. Summarizing the above, we obtain the following:
\begin{thm} \label{exambl}
Let $X$ be a compact complex surface with a K\"ahler form $\a$ and $\pi \colon \Bl_z X \to X$ the blowup of $X$ at a point $z \in X$ with exceptional divisor $E$. Then for any $m>1$, a small deformation of the pair $(\pi^\ast [\a], m \pi^\ast [\a])$ gives rise to a pair of K\"ahler classes $(\Upsilon, \Gamma):=(\pi^\ast [\a]-s[E], m \pi^\ast [\a]-t(s)[E])$ with the associated constant argument $\hat{\Psi}=\hat{\Theta}(s,t(s))>\frac{\pi}{2}$ such that $\cot \hat{\Psi} \Upsilon+\Gamma$ is semipositive, but not K\"ahler. Moreover, for any K\"ahler form $\b \in \Upsilon$ and smooth representative $\tilde{F} \in \Gamma$, we can construct a family of K\"ahler metrics $F_\d=F_{\dHYM,\d}+\d \tilde{F} \in \Gamma$ whose Lagrangian phase $\Theta_\b(F_\d)$ satisfies the hypercritical phase condition.
\end{thm}
%==============Section 5==========================
\section{Convergence of the flow} \label{convflow}
Let $X$ be a compact complex surface with a K\"ahler form $\a$ and a closed real $(1,1)$-form $\hat{F}$ such that the associated constant argument $\hat{\Theta}$ satisfies $\hat{\Theta}>\frac{\pi}{2}$ and $\hat{\Omega}:=\cot \hat{\Theta} \a+\hat{F} \geq 0$.
\subsection{Degenerate complex Monge--Amp\`ere equation}
 We consider the following degenerate complex Monge--Amp\`ere equation
\begin{equation} \label{dgMA}
(\hat{\Omega}+\dd \psi)^2=(1+\cot^2 \hat{\Theta}) \a^2.
\end{equation}
By \eqref{decot}, one can easily check that the volume of $\hat{\Omega}$ is given by $(1+\cot^2 \hat{\Theta})[\a]^2$ as follows:
\begin{eqnarray*}
[\hat{\Omega}]^2 &=& (1+\cot^2 \hat{\Theta}) [\a]^2+2 \cot \hat{\Theta} [\a] \cdot [\hat{F}]+[\hat{F}]^2-[\a]^2 \\
&=& (1+\cot^2 \hat{\Theta}) [\a]^2>0.
\end{eqnarray*}
Also we have
\begin{eqnarray*}
[\hat{\Omega}] \cdot [\hat{F}] &=& \cot \hat{\Theta} [\a] \cdot [\hat{F}]+[\hat{F}]^2\\
&=& \frac{1}{2}([\a]^2+[\hat{F}]^2)>0.
\end{eqnarray*}
Thus we can apply \cite[Proposition 4.5]{SW08} to see that there exists a finite number $N \geq 0$, irreducible curves $C_i$ with $[C_i]^2<0$ on $X$ and numbers $a_i>0$ such that $[\hat{\Omega}]-\sum_{i=1}^N a_i [C_i]$ is a K\"ahler class. In particular, there exist Hermitian metrics $h_i$ on the holomorphic line bundles associated to $C_i$ such that
\[
\hat{\Omega}-\sum_{i=1}^N a_i R_{h_i}>0,
\]
where $R_{h_i}:=-\sqrt{-1} \p \bp \log |s_i|_{h_i}^2$ denotes the curvature of $h_i$. From the same argument as in \cite[Section 2]{FLSW14} (based on \cite{EGZ09} and \cite{Tsu88}), we obtain the following:
\begin{thm} \label{solvdgMA}
There exists a unique bounded $\hat{\Omega}$-PSH function $\psi$ on $X$ (up to additive constant) with
\[
(\hat{\Omega}+\dd \psi)^2=(1+\cot^2 \hat{\Theta}) \a^2
\]
in the sense of currents. Moreover, $\psi$ is smooth away from $\bigcup_i C_i$.
\end{thm}
As argued in the proof of \cite[Theorem 1.2]{JY17}, the solution $\psi$ satisfies \eqref{dHYM} on $X \backslash \bigcup_i C_i$ and vice versa (however, in order to get the solution on the whole space $X$, the strict positivity of $\hat{\Omega}$ is needed).
\subsection{$C^0$-estimate}
In what follows, we assume $\phi_0 \in \cH_{\HC}$, and consider the LBMCF $\phi_t$ ($t \in [0,\infty)$) starting from $\phi_0$. In later argument, we denote constants depending only on $\a$ and $\phi_0$ by the same $c$ or $C$, but it changes from line to line.
\begin{lem} \label{zoest}
There exists a uniform constant $C>0$ such that $\|\phi_t\|_{C^0} \leq C$.
\end{lem}
\begin{proof}
Since $\Theta_\a, \hat{\Theta} \in (\frac{\pi}{2}, \pi)$, on $X \backslash \bigcup_i C_i$ we compute
\begin{eqnarray*}
\ddt \phi_t &=& \Theta_\a-\hat{\Theta}\\
&=& \pi+\arctan \frac{\Im \zeta}{\Re \zeta}-\bigg( \pi+\arctan \frac{\sin \hat{\Theta}}{\cos \hat{\Theta}} \bigg) \\
&=& \arctan \frac{\frac{\Im \zeta}{\Re \zeta}-\frac{\sin \hat{\Theta}}{\cos \hat{\Theta}}}{1+\frac{\Im \zeta}{\Re \zeta} \cdot \frac{\sin \hat{\Theta}}{\cos \hat{\Theta}}}.
\end{eqnarray*}
The numerator in $\arctan (\cdot)$ is computed as
\begin{eqnarray*}
\frac{\Im \zeta}{\Re \zeta}-\frac{\sin \hat{\Theta}}{\cos \hat{\Theta}} &=& \frac{1}{\Re \zeta \sin \hat{\Theta} \cos \hat{\Theta}}(\sin \hat{\Theta} \cos \hat{\Theta} \Im \zeta-\sin^2 \hat{\Theta} \Re \zeta)\\
&=& \frac{1}{\Re \zeta \sin \hat{\Theta} \cos \hat{\Theta}} \frac{2 \cot \hat{\Theta} \a \wedge F_\phi-(\a^2-F_\phi^2)}{(1+\cot^2 \hat{\Theta}) \a^2}\\
&=& \frac{1}{\Re \zeta \sin \hat{\Theta} \cos \hat{\Theta}} \frac{(\cot \hat{\Theta} \a+F_\phi)^2-(1+\cot^2 \hat{\Theta})\a^2}{(1+\cot^2 \hat{\Theta}) \a^2}\\
&=& \frac{1}{\Re \zeta \sin \hat{\Theta} \cos \hat{\Theta}} \bigg( \frac{(\hat{\Omega}+\dd \phi)^2}{(\hat{\Omega}+\dd \psi)^2}-1 \bigg),
\end{eqnarray*}
where $\psi$ denotes the solution to the degenerate complex Monge--Amp\`ere equation \eqref{dgMA}. So we have
\[
\ddt \phi_t=\arctan \bigg[ \frac{1}{\Re \zeta \sin \hat{\Theta} \cos \hat{\Theta}+\Im \zeta \sin^2 \hat{\Theta}} \bigg( \frac{(\hat{\Omega}+\dd \phi)^2}{(\hat{\Omega}+\dd \psi)^2}-1 \bigg) \bigg].
\]
Since $\phi_0 \in \cH_{\HC}$, Lemma \ref{hpufl} shows that $\Re \zeta<0$ and $\Im \zeta=\lambda_1+\lambda_2 \geq c$ for some constant $c>0$ depending only on $\a$ and $\phi_0$. Thus we get a uniform lower bound
\[
\Re \zeta \sin \hat{\Theta} \cos \hat{\Theta}+\Im \zeta \sin^2 \hat{\Theta} \geq c \sin^2 \hat{\Theta}.
\]
Now we will use the trick in the proof of \cite[Proposition 2.2]{FLSW14}. For any $\e>0$, we set
\[
\Phi_\e:=\phi_t-(1+\e)\psi+\e \sum_{i=1}^N a_i \log|s_i|_{h_i}^2-\arctan \bigg(\frac{\e}{A} \bigg)t,
\]
where $A>0$ is a uniform constant (independent of $\e$ and $T$) determined later. Observe that $\Phi_\e$ is smooth on $X \backslash \bigcup_i C_i$, and tends to negative infinity along $\bigcup_i C_i$. Hence for each time $t$, $\Phi_\e$ achieves a maximum in $X \backslash \bigcup_i C_i$. Let $(\hat{x}, \hat{t})$ be a maximum point of $\Phi_\e$. Then at $(\hat{x}, \hat{t})$ we have $\dd \Phi_\e \leq 0$, which yields that
\begin{eqnarray*}
\hat{\Omega}+\dd \phi &=& (1+\e)(\hat{\Omega}+\dd \psi)-\e \bigg( \hat{\Omega}-\sum_{i=1}^N a_i R_{h_i} \bigg)+\dd \Phi_\e \\
&\leq& (1+\e)(\hat{\Omega}+\dd \psi),
\end{eqnarray*}
where we used the fact that $\hat{\Omega}-\sum_{i=1}^N a_i R_{h_i}$ is K\"ahler in the last inequality. Although we do not know whether $\hat{\Omega}+\dd \phi \geq 0$ at $(\hat{x}, \hat{t})$, we have a uniform lower bound
\[
\hat{\Omega}+\dd \phi=\cot \hat{\Theta} \a+F_\phi>\cot \hat{\Theta} \a
\]
since the K\"ahler condition $F_\phi>0$ is preserved under the flow. Let $\mu_i$ be the eigenvalues of $\hat{\Omega}+\dd \phi$ with respect to $\a$ at $(\hat{x}, \hat{t})$. There are three cases:
\begin{enumerate}
\item $\mu_1 \cdot \mu_2 \leq 0$.
\item $\mu_1 \leq 0$ and $\mu_2 \leq 0$.
\item $\mu_1 \geq 0$ and $\mu_2 \geq 0$.
\end{enumerate}
The first case is easy since we have $(\hat{\Omega}+\dd \phi)^2=\mu_1 \mu_2 \a^2 \leq 0$. In the second case, we have $(\hat{\Omega}+\dd \phi)^2 \leq \cot^2 \hat{\Theta} \a^2$ and
\[
\frac{(\hat{\Omega}+\dd \phi)^2}{(\hat{\Omega}+\dd \psi)^2}-1 \leq \frac{\cot^2 \hat{\Theta}\a^2}{(1+\cot^2 \hat{\Theta})\a^2}-1=-\frac{1}{1+\cot^2 \hat{\Theta}}<0.
\]
In these two cases, the derivative $\ddt \phi_t$ as well as $\ddt \Phi_t$ is clearly negative. In the third case, we have $(\hat{\Omega}+\dd \phi)^2 \leq (1+\e)^2(\hat{\Omega}+\dd \psi)^2$ and
\[
\frac{(\hat{\Omega}+\dd \phi)^2}{(\hat{\Omega}+\dd \psi)^2}-1 \leq (1+\e)^2 \frac{(\hat{\Omega}+\dd \psi)^2}{(\hat{\Omega}+\dd \psi)^2}-1=\e^2+2\e.
\]
Thus we get an estimate
\begin{eqnarray*}
\ddt \Phi_\e &=& \ddt \phi_t-\arctan \bigg( \frac{\e}{A} \bigg)\\
&\leq& \arctan \bigg( \frac{\e^2+2\e}{\Re \zeta \sin \hat{\Theta} \cos \hat{\Theta}+\Im \zeta \sin^2 \hat{\Theta}} \bigg)-\arctan \bigg( \frac{\e}{A} \bigg)\\
&\leq& \arctan \bigg( \frac{\e^2+2\e}{c\sin^2 \hat{\Theta}} \bigg)-\arctan \bigg( \frac{\e}{A} \bigg).
\end{eqnarray*}
This computation tells us that we should take $A=\frac{c}{3} \sin^2 \hat{\Theta}$. Summarizing the above, in all cases, we have
\[
\ddt \Phi_\e<0
\]
at $(\hat{x}, \hat{t})$ for all $\e \in (0,1)$, which implies that $\hat{t}=0$. Since the constant $A$ does not depend on $\e$ and $\psi$ is bounded, by letting $\e \to 0$ we have a uniform bound of $\phi_t$ as desired. In order to get a lower bound of $\phi_t$, we consider the minimum of the function
\[
\Psi_\e:=\phi_t-(1-\e)\psi-\e \sum_{i=1}^N a_i \log|s_i|_{h_i}^2+\arctan \bigg(\frac{\e}{B} \bigg)t
\]
for some uniform constant $B>0$. We observe that for each time $t$, $\Psi_\e$ achieves a minimum in $X \backslash \bigcup_i C_i$. Let $(\hat{x}, \hat{t})$ be a minimum point of $\Psi_\e$. Then at $(\hat{x}, \hat{t})$ we have $\dd \Psi_\e \geq 0$, which yields that
\begin{eqnarray*}
\hat{\Omega}+\dd \phi &=& (1-\e)(\hat{\Omega}+\dd \psi)+\e \bigg( \hat{\Omega}-\sum_{i=1}^N a_i R_{h_i} \bigg)+\dd \Psi_\e \\
&\geq& (1-\e)(\hat{\Omega}+\dd \psi).
\end{eqnarray*}
A difference from the estimate of $\Phi_\e$ is that this immediately implies
\[
(\hat{\Omega}+\dd \phi)^2 \geq (1-\e)^2(\hat{\Omega}+\dd \psi)^2
\]
since $\psi$ is $\hat{\Omega}$-PSH. Then the remaining part is similar.
\end{proof}
\subsection{Estimate for the eigenvalues}
\begin{lem} \label{ucontv}
There exist uniform constants $A, C>0$ such that
\[
v \leq \frac{C}{|s_1|_{h_1}^{2a_1 A} \cdots |s_N|_{h_N}^{2a_N A}}
\]
on $X \backslash \bigcup_i C_i$.
\end{lem}
\begin{proof}
We set $\tilde{F}:=\hat{F}-\sum_{i=1}^N a_i R_{h_i}$ so that $\cot \hat{\Theta} \a+\tilde{F}>0$ from the assumption. In particular, $\hat{F}$ is K\"ahler and the inequality
\[
\cot \hat{\Theta} \a+\tilde{F}>\e \a
\]
holds for sufficiently small $\e>0$. On $X \backslash \bigcup_i C_i$, we compute the evolution equation of $\phi_t-\sum_{i=1}^N a_i \log |s_i|_{h_i}^2$ in the normal coordinates as
\begin{eqnarray*}
\bigg( \ddt-\Delta_\eta \bigg) \bigg(\phi_t-\sum_{i=1}^N a_i \log |s_i|_{h_i}^2 \bigg) &=& \Theta_\a-\hat{\Theta}-\eta^{j \bar{j}} \bigg( \phi-\sum_{i=1}^N a_i \log |s_i|_{h_i}^2 \bigg)_{j \bar{j}} \\
&=& \Theta_\a-\hat{\Theta}-\eta^{j \bar{j}} (F_{j \bar{j}}-\hat{F}_{j \bar{j}}+\hat{F}_{j \bar{j}}-\tilde{F}_{j \bar{j}})\\
&=& \Theta_\a-\hat{\Theta}+\eta^{j \bar{j}}(\tilde{F}_{j \bar{j}}-F_{j \bar{j}})\\
& \geq & \Theta_\a-\hat{\Theta}-\eta^{j \bar{j}}(\cot \hat{\Theta} \a_{j \bar{j}}+F_{j \bar{j}})+\e \eta^{j \bar{j}} \a_{j \bar{j}}.
\end{eqnarray*}
In order to get the estimate for the last line, we use the following lemma, which is inspired by the estimate for parabolic $C$-subsolutions \cite[Lemma 3]{PT17} (or \cite[Proposition 5]{Sze18} in the elliptic case):
\begin{lem}
For $\g,\d>0$, $\hat{\theta} \in [\frac{\pi}{2}+\g, \pi-\g]$, set
\[
D:=\bigg\{(x_1,x_2) \in \R^2 \bigg| \frac{\pi}{2}+\g \leq \arctan x_1+\arctan x_2 \leq \pi-\g \bigg\},
\]
\[
G(x_1,x_2):=\sum_p \arctan x_p-\hat{\theta}-\sum_p \frac{\cot \hat{\theta}+x_p}{1+x_p^2}+\sum_p \frac{\d}{1+x_p^2}, \quad (x_1,x_2) \in D.
\]
Then there exist $\kappa, R>0$ (depending only on $\g$, $\d$ and $\hat{\theta}$) such that $G(x_1,x_2)>\kappa$ for all $(x_1,x_2) \in D \cap \{x_1 \geq R\}$.
\end{lem}
\begin{proof} \label{concav}
We will show this by contradiction. So suppose that for $\kappa_\ell \to 0$, $R_\ell \to \infty$, there exists $(x_{1,\ell}, x_{2,\ell}) \in D \cap \{x_1 \geq R_\ell \}$ such that $G(x_{1,\ell},x_{2,\ell}) \leq \kappa_\ell$ holds. Since $R_\ell \to \infty$, we observe that $x_{1,\ell} \to \infty$. On the other hand, since
\[
\g<\arctan x_{2,\ell}<\pi-\arctan x_{1,\ell}-\g \to \frac{\pi}{2}-\g,
\]
by passing to a subsequence, we may assume $x_{2,\ell} \to x_{2,\infty}$. Put $\theta:=\frac{\pi}{2}+\arctan x_{2,\infty} \in [\frac{\pi}{2}+\g,\pi-\g]$. Since $x_{2,\infty}=\tan(\theta-\frac{\pi}{2})=-\cot \theta$, by taking the limit $\ell \to \infty$ in $G(x_{1,\ell},x_{2,\ell}) \leq \kappa_\ell$, we get
\[
\frac{\pi}{2}+\arctan x_{2,\infty}-\hat{\theta}-\frac{\cot \hat{\theta}+x_{2,\infty}}{1+x_{2,\infty}^2} \leq -\frac{\d}{1+x_{2,\infty}^2}.
\]
By using the equation $\theta=\frac{\pi}{2}+\arctan x_{2,\infty}$, we see that
\[
(1+\cot^2 \theta)(\theta-\hat{\theta})+\cot \theta-\cot \hat{\theta} \leq -\d.
\]
On the other hand, since $\frac{d}{dy} \cot y=-(1+\cot^2 y)$ and $\cot (\cdot)$ is concave on $[\frac{\pi}{2}+\g, \pi-\g]$, the LHS must be non-negative, that yields a contradiction.
\end{proof}
Applying the above lemma, we obtain uniform constants $\kappa, R>0$ depending only on $\hat{\Theta}$, $\e$, $\a$ and $\phi_0$. Now let us consider the maximum of the function
\[
\Phi:=\log v-A \bigg( \phi_t-\sum_{i=1}^N a_i \log |s_i|_{h_i}^2 \bigg),
\]
where the uniform constant $A>0$ is determined in the last part of the proof. Note that $\Phi$ is smooth on $X \backslash \bigcup_i C_i$, and tends to negative infinity along $\bigcup_i C_i$. Hence for each time $t$, $\Phi$ achieves a maximum in $X \backslash \bigcup_i C_i$. Let $(\hat{x}, \hat{t})$ be a maximum point of $\Phi$. Without loss of generality, we may assume $\lambda_1 \geq \lambda_2$ and $\lambda_1(\hat{x}, \hat{t}) \geq R$. Then by Lemma \ref{concav}, we have
\begin{eqnarray*}
\bigg( \ddt-\Delta_\eta \bigg) \bigg( \phi_t-\sum_{i=1}^N a_i \log |s_i|_{h_i}^2 \bigg) &\geq& \Theta_\a-\hat{\Theta}-\sum_p \frac{\cot \hat{\Theta}+\lambda_p}{1+\lambda_p^2}+\sum_p \frac{\e}{1+\lambda_p^2} \\
&\geq& \kappa
\end{eqnarray*}
at $(\hat{x}, \hat{t})$. Subtracting this from the evolution equation of $\log v$ (computed in the proof of Lemma \ref{sderiv}), we get
\[
\bigg( \ddt-\Delta_\eta \bigg) \Phi \leq Cn(n-1)-A \kappa.
\]
So if we set $A:=\frac{Cn(n-1)+1}{\kappa}$, we obtain $\big( \ddt-\Delta_\eta \big) \Phi<0$ at $(\hat{x}, \hat{t})$, and hence $\hat{t}=0$. Combining with the uniform bound of $\phi_t$, we obtain the desired statement.
\end{proof}
\subsection{Higher order estimates and the completion of the proof of Theorem \ref{convf}}
Lemma \ref{ucontv} gives a uniform two-sided control of $F$ away from $\bigcup_i C_i$. As in the proof of Theorem \ref{ltime}, we can therefore apply the standard local theory to \eqref{LBMCF} and obtain the following:
\begin{thm} \label{unifphi}
Let $X$ be a compact complex surface with a K\"ahler form $\a$ and a closed real $(1,1)$-form $\hat{F}$ satisfying $\cot \hat{\Theta} \a+\hat{F} \geq 0$. Let $\phi_t$ be the line bundle mean curvature flow with $\phi_0 \in \cH_{\HC}$. Then for any compact subset $K \subset X \backslash \bigcup_i C_i$ and $k \geq 0$, there exists a constant $C=C(k,K)$ such that
\[
\| \phi_t \|_{C^k(K)} \leq C
\]
for all $t \in [0,\infty)$.
\end{thm}
Moreover, we can show the following:
\begin{lem} \label{pcpha}
We have the pointwise convergence $|d \Theta_\a|_\eta^2 \to 0$ on $X \backslash \bigcup_i C_i$ as $t \to \infty$.
\end{lem}
\begin{proof}
As shown in \cite[Proposition 3.4]{JY17}, the time derivative of the volume functional $V$ is given by
\[
\ddt V(\phi_t)=-\int_X |d \Theta_\a|_\eta^2 v \a^2.
\]
By integrating in $t$ we get
\begin{equation} \label{ufint}
\int_0^\infty \int_X |d \Theta_\a|_\eta^2 v \a^2 dt=V(\phi_0)-\lim_{t \to \infty}V(\phi_t) \leq V(\phi_0)
\end{equation}
since $V (\cdot)$ is non-negative. In the same way as in the proof of \cite[Theorem 1.1]{FLSW14}, we can prove by showing a contradiction. So we suppose that the statement does not hold. Then there exists $x \in X \backslash \bigcup_i C_i$, $\r>0$ and a sequence of times $t_j \to \infty$ such that $|d \Theta_\a|_\eta^2(x,t_j)>\r$ for all $j$. On the other hand, since we have uniform bounds for $|d \Theta_\a|_\eta^2$ and all its time and space derivatives away from $\bigcup_i C_i$, there exists a neighborhood $U$ of $x$ with $\overline{U} \subset X \backslash \bigcup_i C_i$ and $\tau>0$ (independent of $j$) such that $|d \Theta_\a|_\eta^2>\frac{\r}{2}$ on $U \times [t_j,t_j+\tau]$. Also, a uniform lower bound for the eigenvalues $\lambda_i$ yields $v>c$ on $X$ for some uniform constant $c>0$. Thus
\[
\int_0^\infty \int_X |d \Theta_\a|_\eta^2 v \a^2 dt \geq \sum_j \int_{t_j}^{t_j+\tau} \int_U |d \Theta_\a|_\eta^2 v \a^2 dt \geq \sum_j \frac{\r c \tau}{2} \int_U \a^2=\infty.
\]
This contradicts \eqref{ufint}.
\end{proof}
In order to prove the convergence on the level of potentials and Corollary \ref{Jbound}, we need to study the limiting behavior of the functionals $\cI$, $\cJ$ along the flow:
\begin{lem} \label{conve}
Let $\{\varphi_j \} \subset \cH_{\HC}$ and $\varphi_\infty$ be a bounded $\hat{F}$-PSH function on $X$. Let $Y$ be a proper subvariety of $X$. Suppose that
\begin{enumerate}
\item there exists $C>0$ such that $\|\varphi_j\|_{C^0} \leq C$.
\item $\varphi_j \to \varphi_\infty$ in $C_{\rm loc}^\infty(X \backslash Y)$ as $j \to \infty$.
\end{enumerate}
Then the quantities $\cI(\varphi_\infty)$ and $\cJ(\varphi_\infty)$ are well-defined, and the convergence $\cI(\varphi_j) \to \cI(\varphi_\infty)$, $\cJ(\varphi_j) \to \cJ(\varphi_\infty)$ hold as $j \to \infty$.
\end{lem}
\begin{proof}
The definition of $CY_\C(\varphi)$ includes the quantities $F_\varphi^2$, $\hat{F} \wedge F_\varphi$, $F_\varphi \wedge \a$. For a bounded $\hat{F}$-PSH function $\varphi_\infty$, we can define these quantities as finite measures on $X$ which do not charge pluripolar subsets (\cf \cite{BT82}). So the quantity $CY_\C(\varphi_\infty)$ (and hence $\cI(\varphi_\infty)$ and $\cJ(\varphi_\infty)$) is well-defined. Along the same line as in the proof of \cite[Lemma 3.2]{FLSW14}, the convergence properties $\cI(\varphi_j) \to \cI(\varphi_\infty)$, $\cJ(\varphi_j) \to \cJ(\varphi_\infty)$ follow from the convergence $F_{\varphi_j} \to F_{\varphi_\infty}$ in $C_{\rm loc}^\infty(X \backslash Y)$ since $Y$ is a pluripolar subset.
\end{proof}
Now we are ready to prove Theorem \ref{convf}.
\begin{proof}[Proof of Theorem \ref{convf}]

Now we invoke Lemma \ref{zoest} and Theorem \ref{unifphi} to find that there exists a sequence of times $t_\ell \to \infty$ and a bounded function $\phi_\infty$ on $X$ such that $\phi_{t_\ell} \to \phi_\infty$ in $C_{\rm loc}^\infty(X \backslash \bigcup_i C_i)$. Also, from a uniform lower bound $F_{\phi_t} \geq c \a$, we know that $F_{\phi_\infty}:=\hat{F}+\dd \phi_\infty$ is a K\"ahler current, which is smooth on $X \backslash \bigcup_i C_i$. By Lemma \ref{conve}, we have $\cI(\phi_{t_\ell}) \to \cI(\phi_\infty)$. So combining with the monotonicity of $\cI$ (\cf Proposition \ref{monotonicity}), we find that $\cI(\phi_t)$ is bounded from below and converges to $m:=\lim_{t \to \infty} \cI(\phi_t)$ (where the number $m$ is uniquely determined by the initial data $\phi_0$, and independent of a choice of subsequences $\{\phi_{t_\ell}\}$). Moreover, by Lemma \ref{pcpha}, the function $\phi_\infty$ satisfies $d\Theta_\a(\phi_\infty)=0$. Since $\bigcup_i C_i$ has real codimension $2$ and the complement $X \backslash \bigcup_i C_i$ is connected, this yields that $\Theta_\a(\phi_\infty)$ is constant and
\begin{equation} \label{cmap}
\Im \bigg( e^{-\sqrt{-1} \Theta_\a (\phi_\infty)} (\a+\sqrt{-1} F_{\phi_\infty})^2 \bigg)=0
\end{equation}
on $X \backslash \bigcup_i C_i$. The form $(\a+\sqrt{-1} F_\varphi)^2$ includes the quantities $F_\varphi^2$, $\a \wedge F_\varphi$. So in the same way as in the previous lemma, we can successfully define these quantities for the bounded $\hat{F}$-PSH function $\phi_\infty$, as finite measures on $X$ which do not charge pluripolar subsets. Thus the equality \eqref{cmap} holds on $X$ in the sense of currents. By integrating on $X$ and using the cohomological condition, we get $\Theta_\a(\phi_\infty)=\hat{\Theta}$. Similarly, one can check that $\phi_\infty$ is a $\hat{\Omega}$-PSH function satisfying
\[
(\hat{\Omega}+\dd \phi_\infty)^2=(1+\cot^2 \hat{\Theta}) \a^2
\]
in the sense of currents by using the dHYM condition just as in the proof of \cite[Theorem 1.2]{JY17}. Hence we know that $\phi_\infty$ is the unique solution to \eqref{dgMA} with $\cI(\phi_\infty)=m$ (\cf Theorem \ref{solvdgMA}).

We remark that all of the above arguments still hold for all subsequences of $\phi_t$. Thus we conclude that $\phi_t$ converges to the unique solution $\phi_\infty$ of \eqref{dgMA} with $\cI(\phi_\infty)=m$ in $C_{\rm loc}^{\infty}(X \backslash \bigcup_i C_i)$ as $t \to \infty$ without taking subsequences. The convergence $F_{\phi_t} \to F_{\phi_\infty}$ on $X$ in the sense of currents follows from the convergence $F_{\phi_t} \to F_{\phi_\infty}$ in $C_{\rm loc}^\infty(X \backslash \bigcup_i C_i)$. This completes the proof.
\end{proof}
\begin{proof}[Proof of Corollary \ref{Jbound}]
Take any $\phi_0 \in \cH_{\HC}$. By Theorem \ref{convf} and the monotonicity of $\cJ$ (\cf Proposition \ref{monotonicity}), we have
\[
\cJ(\phi_0) \geq \lim_{t \to \infty} \cJ(\phi_t)=\cJ(\phi_\infty),
\]
where the function $\phi_\infty$ is the unique solution to \eqref{dgMA} (modulo constant). So this gives a desired lower bound of $\cJ$ since $\cJ(\phi_\infty+c)=\cJ(\phi_\infty)$ for all $c \in \R$, and hence the quantity $\cJ(\phi_\infty)$ is independent of a choice of $\phi_0 \in \cH_{\HC}$. This completes the proof.
\end{proof}
%==============References==========================

\end{document}